\documentclass[11pt]{amsart}
\usepackage{amssymb,amsmath}
\usepackage{amsthm}
\usepackage{supertabular,float}
\usepackage{graphicx}
\usepackage[latin1]{inputenc}
\usepackage{amsfonts}
\usepackage[english]{babel}
\usepackage{comma}

\newtheorem{theorem}{Theorem}
\theoremstyle{plain}

\newtheorem{corollary}[theorem]{Corollary}

\newtheorem{example}[theorem]{Example}

\newtheorem{lemma}[theorem]{Lemma}

\newtheorem{proposition}[theorem]{Proposition}

\numberwithin{equation}{section} \makeatletter
\newcommand{\figcaption}{\def\@captype{figure}\caption}
\newcommand{\tabcaption}{\def\@captype{table}\caption}
\makeatother

\begin{document}

\title[Midy's Sets]{Structure of associated sets to Midy's Property}
\author[J.H. Castillo]{John H. Castillo}
\address{John H. Castillo, Departamento de Matemáticas y Estadística, Universidad de Nariño}
\email{jhcastillo@gmail.com}
\author[G. García-Pulgarín]{Gilberto Garc\'\i a-Pulgar\'in}
\address{Gilberto Garc\'\i a-Pulgar\'in, Universidad de Antioquia}
\email{gigarcia@ciencias.udea.edu.co}
\author[J.M Velásquez Soto]{Juan Miguel Vel\'asquez-Soto}
\address{Juan Miguel Vel\'asquez Soto, Departamento de Matemáticas, Universidad del Valle}
\email{jumiveso@univalle.edu.co}

\begin{abstract}
Let $b$ be a positive integer greater than $1$,  $N$ a positive
integer relatively prime to $b$, $
\left\vert b\right\vert _{N}$ the order of $b$ in the multiplicative group $%
\mathbb{U}_{N}$ of positive integers less than $N$ and relatively primes to $%
N,$ and $x\in\mathbb{U}_{N}$. It is well known that when we write
the fraction $\frac{x}{N}$ in base $b$, it is periodic. Let $d,\,k$ be positive integers with $%
d\geq2$ and such that $\left\vert b\right\vert _{N}=dk$ and $\frac{x}{N}=0.%
\overline{a_{1}a_{2}\cdots a_{\left\vert b\right\vert _{N}}}$ with
the bar indicating the period and $a_{i}$ are digits in base $b$. We
separate the period ${a_{1}a_{2}\cdots a_{\left\vert b\right\vert
_{N}}}$ in $d$ blocks of length $k$ and let $
A_{j}=[a_{(j-1)k+1}a_{(j-1)k+2}\cdots a_{jk}]_{b} $
be the number represented in base $b$ by the $j-th$ block and $%
S_{d}(x)=\sum\limits_{j=1}^{d}A_{j}$. If for all
$x\in\mathbb{U}_{N}$, the sum $S_{d}(x)$ is a multiple of $b^{k}-1$
we say that $N$ has the Midy's property for $b$ and $d$.

In this work we present some interesting properties of the set of
positive integers $d$ such that $N$ has the Midy's property to for
$b$  and $d$.


\end{abstract}

\maketitle

\section{Introduction}

Let $b$ be a positive integer greater than $1$, $b$ will denote the
base of
numeration, $N$ a positive integer relatively prime to $b$, i.e $(N,b)=1$, $%
\left\vert b\right\vert _{N}$ the order of $b$ in the multiplicative group $%
\mathbb{U}_{N}$ of positive integers less than $N$ and relatively primes to $%
N,$ and $x\in\mathbb{U}_{N}$. It is well known that when we write
the fraction $\frac{x}{N}$ in base $b$, it is periodic. By period we
mean the smallest repeating sequence of digits in base $b$ in such
expansion, it is easy to see that $\left\vert b\right\vert _{N}$ is
the length of the period of the fractions $\frac{x}{N}$ (see
Exercise 2.5.9 in \cite{Nathanson}). Let $d,\,k$ be positive integers with $%
d\geq2$ and such that $\left\vert b\right\vert _{N}=dk$ and $\frac{x}{N}=0.%
\overline{a_{1}a_{2}\cdots a_{\left\vert b\right\vert _{N}}}$ with
the bar indicating the period and $a_{i}$ are digits in base $b$. We
separate the period ${a_{1}a_{2}\cdots a_{\left\vert b\right\vert
_{N}}}$ in $d$ blocks of length $k$ and let
\begin{equation*}
A_{j}=[a_{(j-1)k+1}a_{(j-1)k+2}\cdots a_{jk}]_{b}
\end{equation*}
be the number represented in base $b$ by the $j-th$ block and $%
S_{d}(x)=\sum\limits_{j=1}^{d}A_{j}$. If for all
$x\in\mathbb{U}_{N}$, the sum $S_{d}(x)$ is a multiple of $b^{k}-1$
we say that $N$ has the Midy's property for $b$ and $d$. It is named
after E. Midy (1836), to read historical aspects about this property
see \cite{Lewittes} and its references.

If $D_{b}(N)$ is the number in base $b$ represented by the period of $\frac {%
1}{N}$, this is $D_{b}(N)=[a_{1}a_{2}\cdots a_{\left\vert
b\right\vert _{N}}]_{b},$ it is easy to see that
$ND_{b}(N)=b^{\left\vert b\right\vert _{N}}-1$. We denote with
$\mathcal{M}_{b}(N)$ the set of positive integers $d$ such that $N$
has the Midy's property for $b$ and $d$ and we will call it the
Midy's set of $N$ to base $b$. As usual, let $\nu_{p}(N)$ be the
greatest exponent of $p$ in the prime factorization of $N$.

For example $13$ has the Midy's property to the base $10$ and $d=3$,
because $|13|_{10}=6$, $1/13=0.\overline{076923}$ and $07+69+23=99$.
Also, $49$ has the Midy's property to the base $10$ and $d=14$,
since $|49|_{10}=42$,
$1/49=0.\overline{020408163265306122448979591836734693877551}$ and
$020+408+163+265+306+122+448+979+591+836+734+693+877+551=7*999$. But
$49$ does not have the Midy's property to $10$ and $7$. Actually, we
can see that $\mathcal{M}_{10}(13)=\{2,3,6\}$ and
$\mathcal{M}_{10}(49)=\{2, 3, 6, 14, 21, 42\}$.

In \cite{garcia09} are given the following characterizations of the
Midy's pro\-per\-ty.

\begin{theorem}
\label{d}Let $N,b$ and $d$ as above, $d\in \mathcal{M}_{b}(N)$ if
and only if $D_{b}(N)\equiv 0 \pmod{b^{k}-1}$. Furthermore, if $d\in
\mathcal{M}_{b}(N)$ and $D_{b}(N)=(b^{k}-1)t$, for some integer $t$,
then $b^{|b|_N}-1=(b^{k}-1)Nt$.
\end{theorem}

\begin{theorem}
\label{ppl}Let $N,b$ and $d$ as above, $d\in \mathcal{M}_{b}(N)$ if
and only if for all prime $p$ divisor of $N$ it satisfies that if
$|b|_{p}\mid k$, then $\nu_{p}(N)\leq\nu_{p}(d)$. Furthermore, if
$d\in \mathcal{M}_{b}(N)$, then $\sum\nolimits_{i=1}^{d}(b^{ik}
\mod{N})=m_{b}(d,\ N)N$.
\end{theorem}

\begin{theorem}
\label{ppl2}Let $N,b$ and $d$ as above, $d\in \mathcal{M}_{b}(N)$ if and only if for all prime $p$ divisor of $%
(b^{k}-1,N)$ it satisfies that $\nu_{p}(N)\leq\nu_{p}(d)$.
\end{theorem}

\section{Structure of $\mathcal{M}_{b}(N)$}

Theorem \ref{ppl} tells us that the subgroup generated by $b^{k}$ in $%
\mathbb{U}_{N}$, $\left\langle b^{k}\right\rangle =\linebreak
\left\{ b^{jk}:j=0,1,\ldots,\ d-1\right\}$; is the key of a method
to obtain the value of the multiplier $m_{b}(d,\ N)$, because if
$d\in \mathcal{M}_{b}(N)$, then
\begin{equation*}
Nm_{b}(d,\ N)=\sum\nolimits_{i=1}^{d}(b^{ik} \mod N)\text{.}
\end{equation*}

The following result shows an interesting relationship between
$\left\langle
b^{k_{2}}\right\rangle $ and $\left\langle b^{k_{1}}\right\rangle $ when $%
k_{2}\mid k_{1}$.

\begin{theorem}
\label{d_divide}If $\left\vert b\right\vert _{N}=k_{1}d_{1}=k_{2}d_{2}$ and $%
d_{2}=cd_{1}$ for some integer $c\in\mathbb{Z}$; then
\begin{equation*}
\left\langle b^{k_{2}}\right\rangle
=\bigcup\limits_{r=0}^{c-1}\left( b^{rk_{2}}\left\langle
b^{k_{1}}\right\rangle \right)
\end{equation*}
where $b^{rk_{2}}\left\langle b^{k_{1}}\right\rangle =\left\{
b^{rk_{2}}x:x\in\left\langle b^{k_{1}}\right\rangle \right\} $.
\end{theorem}

\begin{proof}
Since $d_{2}=cd_{1}$ the $d_{2}$ values of $j\in\{0,\ 1,\ \ldots,\
d_{2}-1\}$ can be divided by $c$ obtaining a quotient between $0$
and $d_{1}-1$ and a remainder between $0$ and $c-1$, in consequence
this values are the numbers
$ci+r$ with $0\leq i\leq d_{1}-1$ and $0\leq r\leq c-1$. Thus%

\begin{align*}
\left\langle b^{k_{2}}\right\rangle  &  =\left\{  b^{jk_{2}}:j=0, 1,
\ldots,
d_{2}-1\right\} \\
&  =\left\{  b^{k_{2}(ci+r)}:i=0, 1, \ldots,\ d_{1}-1,\ \ r=0,1,
\ldots,
c-1\right\} \\
&  =\left\{  b^{k_{1}i+rk_{2}}:i=0,1, \ldots, d_{1}-1, \ r=0,\ 1,
\ldots,\ c-1\right\} \\
&  =\bigcup\limits_{r=0}^{c-1}\left(  b^{rk_{2}}\left\langle b^{k_{1}%
}\right\rangle \right)
\end{align*}

\end{proof}

We get the following result as a consequence of the above fact.

\begin{corollary}
\label{coro}Let $d_{1}$, $d_{2}$ be divisors of $\left\vert
b\right\vert _{N}
$ and assume that $d_{1}\mid d_{2}$ and $d_{1}\in \mathcal{M}_{b}(N)$, then $%
d_{2}\in \mathcal{M}_{b}(N)$.
\end{corollary}

The following result is a dual version of this corollary.

\begin{proposition}
Let $N_{1}$, $N_{2}$ and $d$ be integers such that $d$ is a common
divisor of $\left\vert b\right\vert _{N_{1}}$ and $\left\vert
b\right\vert _{N_{2}}$, if $d\in \mathcal{M}_{b}(N_{2})$ and
$N_{1}\mid N_{2}$ then $d\in \mathcal{M}_{b}(N_{1})$.
\end{proposition}

\begin{proof}
In fact, as $N_{1}\mid N_{2}$, if $\left\vert b\right\vert
_{N_{2}}=k_{2}d$ then $\left\vert b\right\vert _{N_{1}}=k_{1}d$ with
$k_{1}\mid k_{2}$. Thus $\left(  b^{k_{1}}-1,\ N_{1}\right)
\mid\left(  b^{k_{2}}-1,\ N_{2}\right)  $ and the result follows
from Theorem \ref{ppl} and from the fact that $d\in
\mathcal{M}_{b}(N_2)$.
\end{proof}

\begin{theorem}
If $2\in \mathcal{M}_{b}(N)$ and $d$ divides $\left\vert
b\right\vert _{N}$ with $\ d$ even, then $d\in \mathcal{M}_{b}(N)$
and $m_{b}(d,\ N)=\frac{d}{2} $.
\end{theorem}

\begin{proof}
In Theorem \ref{d_divide}, letting $d_{1}=2$, $k_{1}=\dfrac{e}{2}$,
$d_{2}=d$ and therefore $c=\dfrac{d}{2}$ and $\left\langle
b^{k_{1}}\right\rangle =\left\{ 1,\ N-1\right\}  $ we obtain that
$\left\langle b^{k_{2}}\right\rangle $ is formed by $c$ translations
of $\left\{ 1,\ N-1\right\}  $ and so the sum of its elements is
$cN$, thus we have $m_{b}(d,\ N)=c=\dfrac{d}{2}$.
\end{proof}

The hypothesis $2\in \mathcal{M}_{b}(N)$ is essential, as is shown
in the following example due to Lewittes, see \cite{Lewittes}.
\begin{example}
Let $%
N=7\times19\times9901,$ so $\left\vert 10\right\vert _{N}=36$ and,
in addition, $N$ does not have the Midy's property for the base $10$
and for any $d=2,\ 3,\ 6$; but it has this property when $d=4,\ 9,\
12,\ 18$ and $36$ and $m_{10}(12,\ N)=7$.
\end{example}

Next theorem has a big influence in our work.

\begin{theorem}[Theorem 3.6 in \protect\cite{Nathanson}]
\label{poten}Let $p$ be an odd prime not dividing $b$, $%
m=\nu_{p}(b^{\left\vert b\right\vert _{p}}-1)$ and let $t$ be a
positive integer, then
\begin{equation*}
\left\vert b\right\vert _{p^{t}}=%
\begin{cases}
\left\vert b\right\vert _{p} & \text{ if }t\leq m, \\
&  \\
p^{t-m}\left\vert b\right\vert _{p} & \text{ if \ }t>m.%
\end{cases}
\end{equation*}
\end{theorem}

For the base $b=10$ the greatest $m$ known is $2,$ which is achieved
with the primes $3,\ 487$ and $56598313$, see \cite{Montgomery}. \
From the same paper we take the following example: \ if $b=68$ and
$p=113$, then $\ \left\vert b\right\vert _{p}=\ \left\vert
b\right\vert _{p^{2}}=\ \left\vert
b\right\vert _{p^{3}}$. Something similar occurs for $b=42$ and $p=23$. For $%
m=3,$ these are the only cases with $p<2^{32}$ and $2\leq b\leq91$.

Next theorem allows us to build $\mathcal{M}_{b}(p^{n})$ from $\mathcal{M%
}_{b}(p)$.

\begin{theorem}
Let $b,\ p,\ n$ be integers where $p$ is a prime not dividing $b$,
and $n$ positive. Let $m=\nu _{p}(b^{\left\vert b\right\vert
_{p}}-1)$, then
\begin{equation*}
\mathcal{M}_{b}(p^{n})=%
\begin{cases}
\mathcal{M}_{b}(p) & \text{if }n\leq m, \\
\bigcup\limits_{i=0}^{n-m}p^{n-m-i}\mathcal{M}_{b}(p) & \text{if }n>m\text{.}%
\end{cases}%
\end{equation*}%
Therefore;
\begin{equation*}
\left\vert \mathcal{M}_{b}(p^{n})\right\vert =%
\begin{cases}
\left\vert \mathcal{M}_{b}(p)\right\vert  & \text{if }n\leq m, \\
(n-m+1)\left\vert \mathcal{M}_{b}(p)\right\vert  & \text{if }n>m\text{.}%
\end{cases}%
\end{equation*}
\end{theorem}

\begin{proof}
Let $\left\vert b\right\vert _{p}=kd$ and $d\in \mathcal{M}_{b}(p)$ then $(b^{k}%
-1,\ p)=1$. Suppose that $n\leq m$, as $(b^{k}-1,\ p^{n})=1$ and
$\left\vert b\right\vert _{p^{n}}=\left\vert b\right\vert _{p}=kd$
follows that $d\in \mathcal{M}_{b}(p^n)$ and thus
$\mathcal{M}_{b}(p)\subset\mathcal{M}_{b}(p^{n})$. It is also easy
to prove that $\mathcal{M}_{b}(p^{n})\subset\mathcal{M}_{b}(p)$.

We now consider the case when $n>m$. Let $d\in\mathcal{M}_{b}(p)$
and $\left\vert b\right\vert _{p}=kd$, and let $i$ be an integer
with $0\leq i\leq n-m$, by Theorem \ref{poten} we have that
$\left\vert b\right\vert _{p^{n}} =p^{n-m}\left\vert b\right\vert
_{p}=kp^{i}(p^{n-m-i}d)$. We affirm that $(b^{kp^{i} }-1,\ p^{n})=1$
because $b^{kp^{i}}\equiv(b^{k})^{p^{i}}\equiv b^{k} \
\operatorname{mod}\ p\not \equiv 1\ \operatorname{mod}\ p$. As
$(b^{kp^{i} }-1,\ p^{n})=1$ and $\left\vert b\right\vert
_{p^{n}}=kp^{i}(p^{n-m-i}d)$ it follows from Theorem \ref{ppl2} that
$p^{n-m-i}d\in \mathcal{M}_{b}(p^n)$. \ In this way we have proved
that $p^{n-m-i}\mathcal{M}_{b}(p)\subset\mathcal{M}_{b}(p^{n})$.

Similarly, we can show that $\mathcal{M}_{b}(p^{n})\subset
p^{n-m-i}\mathcal{M}_{b}(p)$. The second part of the theorem is a
direct consequence from the first part.
\end{proof}

Theorem \ref{ppl2} says that if $p$ is prime and $d>1$ is a divisor of $%
\left\vert b\right\vert _{p}$, then $d\in \mathcal{M}_{b}(p)$ and therefore $%
\left\vert \mathcal{M}_{b}(p)\right\vert =\tau(o_{p}(b))-1$, where
$\tau(n)$ denote the number of positive divisors of $n$.

\begin{theorem}
\label{john}Let $N$, $M$ be integers such that $\left\vert
b\right\vert _{MN}=\left\vert b\right\vert _{N}$, then
\end{theorem}

\begin{enumerate}
\item $\mathcal{M}_{b}(MN)\subseteq\mathcal{M}_{b}(N)$. \

\item If $N$ and $M$ are relatively primes, then
\begin{equation*}
\mathcal{M}_{b}(MN)=\left\{
\begin{array}{c}
d\in\mathcal{M}_{b}(N):\left\vert b\right\vert _{N}=kd\text{ \ and \
\ }
\\
\forall\left( r\text{ primo}\right) \left( r\mid\left( b^{k}-1\text{, }%
M\right) \Rightarrow\nu_{r}\left( M\right) \leq\nu_{r}\left( d\right) \right)%
\end{array}
\right\} .
\end{equation*}

\item In particular, if $p$ is a prime not dividing $N$, $\left\vert
b\right\vert _{p}$ is a divisor of $\left\vert b\right\vert _{N}$, and $%
s=\nu_{p}(\left\vert b\right\vert _{N})$, then
\begin{equation*}
\mathcal{M}_{b}(p^{s+1}N)=\left\{ d\in\mathcal{M}_{b}(N):\left\vert
b\right\vert _{N}=kd\text{ \ and \ }\left( b^{k}-1\text{, }p\right)
=1\right\} .
\end{equation*}
\end{enumerate}

\begin{proof}
To prove the first part we show that if $d\notin\mathcal{M}_{b}(N)$,
then $d\notin\mathcal{M}_{b}(MN)$. In fact, as $\left\vert
b\right\vert _{N}=\left\vert b\right\vert _{MN}=kd$ and
$d\notin\mathcal{M}_{b}(N)$ from Theorem \ref{ppl2}, there exists a
prime $q,$ divisor of $\left( b^{k}-1,\ N\right)  $ such that
$\nu_{q}\left(  N\right)  >\nu_{q}\left( d\right)  .$ As $\left(
b^{k}-1,\ N\right)$  is a divisor of   $\left(  b^{k}-1,\ MN\right)
$ and $\nu_{q}\left(  MN\right)  \geq\nu_{q}\left(  N\right)  $
Theorem \ref{ppl2} guarantees that $d\notin\mathcal{M}_{b}(MN)$.

We now add the hypothesis $\left(  M,N\right)  =1$ and let
$\left\vert
b\right\vert _{N}=\left\vert b\right\vert _{MN}=kd$ with $d\in\mathcal{M}%
_{b}(N)$. Consider a prime $r$ divisor of $\left(  b^{k}-1,MN\right)
$. Since $M$ and $N$ are relatively primes then either $r\mid\left(
b^{k}-1,M\right) $ or $r\mid\left(  b^{k}-1,N\right)  $, but not
both. If $r\mid\left( b^{k}-1,N\right)  ,$ as
$d\in\mathcal{M}_{b}(N)$ from Theorem \ref{ppl2} follows that
$\nu_{r}\left(  N\right)  \leq\nu_{r}\left(  d\right)  $ and as $M$
and $N$ are relatively primes we have that $\nu_{r}\left(  N\right)
=\nu_{r}\left(  MN\right)  $ and therefore
$d\in\mathcal{M}_{b}(MN)$. If $r\mid\left(  b^{k}-1,M\right)  $, as
$r\nmid N$, we have $\nu_{r}\left( MN\right)  =\nu_{r}\left(
M\right) $ and from the assumption and Theorem \ref{ppl2} we get
that $d\in\mathcal{M}_{b}(MN)$.

The third part now is clear, because $\left\vert b\right\vert
_{p^{s+1}}$ is a divisor of $\left\vert b\right\vert _{N}$ and $p$
and $N$ are relatively primes.
\end{proof}

\begin{theorem}
Let $N,$ $p$ be integers with $(N,b)=1$ with $p$ a prime divisor of
$b-1$. Then there exists a positive integer $s$ such that for all
integer $t,$ with $t>s$, we have that
$\mathcal{M}_{b}(p^{t}N)=\varnothing $.
\end{theorem}

\begin{proof}
Without loss of generality we can suppose that $p$ is not a divisor
of $N$. \ Let $s=\nu _{p}(\left\vert b\right\vert _{N})$,  as
$\left\vert
b\right\vert _{p}=1$ we are in the conditions of the third part of Theorem %
\ref{john} and the result is immediately because $\left( b^{k}-1\text{, }%
p\right) =p$ for any $k$.
\end{proof}

The result of previous theorem is true for any divisor $n$, not
necessarily a prime, of $b-1$. Also note that the value of the
integer $s-\nu_p(N)$ is the smallest that satisfies the theorem because $%
\mathcal{M}_{b}(p^{s-\nu_p(N)}N)$ is non empty by the  second part of Theorem \ref%
{john}.

We now study the following question. Given $N$ and $b$ with
$\mathcal{M}_b(N)\neq \varnothing$, is it possible to find a
positive integer $z$ such that $\mathcal{M}_b(zN)=\{|b|_N\}$ ? The
next result, from \cite{motoseI}, will be useful in the sequel.

\begin{lemma}[Corollary 2 in \cite{motoseI}]\label{motose} Let $b\geq 2$ and
$n\geq 2$. Then there exists a prime $p$ with $n=|b|_p$ in all
except the following pairs: $(n,b)=(2,b^{\gamma}-1)$ or $(6,2)$.
\end{lemma}

To answer the question we will need the following result.

\begin{lemma}\label{conditions}
Let $N$ and $b$ be integers such that $\mathcal{M}_b(N)\neq
\varnothing$. Let $q$ a prime divisor of $|b|_N$. Then there exists
a positive integer $z$ that satisfies the following properties
\begin{enumerate}
\item $|b|_{zN}=|b|_N$,
\item $\mathcal{M}_b(zN)\neq \varnothing$,
\item If $d\in \mathcal{M}_b(zN)$, then $\nu_q(d)=\nu_q(|b|_N)$.
\end{enumerate}
\end{lemma}
\begin{proof}
We will study two cases

\begin{enumerate}
\item Assume that either $q\neq 2$ or $b+1$ is not a power of $2$.

From Lemma \ref{motose} there exists an odd prime $p$ such that
$\left\vert b\right\vert _{p}=q$.

In the sequel, we denote with $c=\nu_{p}\left( N\right)$,
$s=\nu_{p}\left( \left\vert b\right\vert _{N}\right)$ and
$m=\nu_{p}\left( b^{q}-1\right) $.

If $p$ is not a divisor of $N$, from the third part of Theorem
\ref{john}, we have that when $d\in \mathcal{M}_b(zN)$, then
$|b|_N=kd$ and $(b^k-1,p)=1$. Hence if $d\in \mathcal{M}_b(zN)$,
then $\nu_q(d)=\nu_q(|b|_N)$. Thus, in this case, we take
$z=p^{s+1}$. Since $(b-1,zN)=(b-1,N)$ and $|b|_N\in
\mathcal{M}_b(N)$ we have that $|b|_N\in\mathcal{M}_b(zN)$.

From now we suppose that $p$ is a divisor of $N$. Thus $c>0$ and
$N=p^{c}M$ with $M$ non divisible by $p$. We consider the following
cases:

\begin{enumerate}
\item $c\geq s+1$.

Let $d\in\mathcal{M}_{b}(N)$ where $\left\vert b\right\vert
_{N}=kd$, if $p$ divides $b^{k}-1,$ then from Theorem \ref{ppl2} it
follows that $c=\nu_{p}\left(  N\right) \leq\nu_{p}\left( d\right)
\leq s$, which is a contradiction. In consequence, we get that $d\in
\mathcal{M}_b(N)$, implies that $|b|_N=kd$ and
$\nu_q(d)=\nu_q(|b|_N)$ and we take $z=1$.

\item $c<s+1$.

We  consider two subcases, depending if either $q$ is or not a
divisor of $\left\vert b\right\vert _{M}$.

Firstly, we assume that $q\mid\left\vert b\right\vert _{M}$. Since
$\left\vert b\right\vert _{N}=\left[  \left\vert b\right\vert
_{p^{c}}, \left\vert b\right\vert _{M}\right] $ and $\left\vert
b\right\vert _{p^{s+1}M}=\left[  \left\vert b\right\vert _{p^{s+1}},
\left\vert b\right\vert _{M}\right] $ from Theorem \ref{poten},
$|b|_{N}=\left[qp^{\delta},|b|_M\right]$ and
$|b|_{p^{s+1}M}=\left[qp^{\varepsilon},\left|b\right|_M\right]$;
where $\delta=\max(0,c-m)$ and $\varepsilon=\max(0,s-m+1)$.

We claim that $\left\vert b\right\vert _{p^{s+1}M}=\left\vert
b\right\vert _{N}=|b|_M$. In fact, since $|b|_N=\left[qp^{\delta},
|b|_M\right]$, $s=\nu_{p}(|b|_N)$ and $\delta<s$, we obtain that
$\nu_{p}(|b|_M)=s$ and hence $|b|_N=|b|_M$. Also as $\varepsilon\leq
s$, we get that $|b|_{p^{s+1}M}=|b|_M$.

By the third part of Theorem \ref{john} we have that $d\in
\mathcal{M}_{b}(p^{s+1}M)$, implies that $\nu_q(d)=\nu_q(|b|_N)$. So
we take $z=p^{s-c+1}$. Again,  as \linebreak $(b-1,zN)=(b-1,N)$ and
$|b|_N\in \mathcal{M}_b(N)$, then $|b|_N\in\mathcal{M}_b(zN)$.

 Assume that $q\nmid\left\vert b\right\vert _{M}$. Similar as in the above paragraph we
 can show that $|b|_{p^{s+1}M}=|b|_N=q|b|_M$.
 We affirm that
$$\mathcal{M}_b(p^{s+1}M)=\{d'q: d'\in \mathcal{M}_b(M)\}.$$

 Let $d'\in \mathcal{M}_b(M)$ since $|b|_{p^{s+1}M}=k(d'q)$ and $(b^k-1,M)=
\linebreak(b^k-1,p^{s+1}M)$, from Theorem \ref{ppl2}, we get that
 $d'q\in\mathcal{M}_b(p^{s+1}M)$. Therefore, $\{d'q:
 d'\in\mathcal{M}_b(M)\}\subseteq \mathcal{M}_b(p^{s+1}M)$.

 Let
 $d\in \mathcal{M}_b(p^{s+1}M)$. Since $|b|_{p^{s+1}M}=q|b|_M$
 we have that $d$ is either a divisor of $|b|_{M}$ or $d=q$ or $d=d'q$ where $d'>1$ is a divisor of $|b|_M$.
 If $d$ is a divisor of $|b|_M$ with $|b|_M=kd$, then as $p$
divides
 $(b^{kq}-1,p^{s+1}M)$ and $s+1=\nu_p(p^{s+1}M)>\nu_p(d)$ by Theorem
 \ref{ppl2} we obtain that $d\not\in\mathcal{M}_b(p^{s+1}M)$. Now assume that $d=q$. Since $p$ divides
  $|b|_M$ there exists a prime $r$ divisor of $(b^{|b|_M}-1,p^{s+1}M)$, with $r\neq q$. By Theorem \ref{ppl2} we get a
  contradiction.

  Finally if $d=d'q$ with $|b|_M=kd'$, it is easy to see that
  $d\in\mathcal{M}_b(p^{s+1}M)$ implies that $d'\in
  \mathcal{M}_b(M)$.

 Thus, in this case we take
 $z=p^{s-c+1}$. We showed that if $d\in \mathcal{M}_b(zN)$,
 then $d=d'q$ where $|b|_N=kd$, $d'\in \mathcal{M}_b(M)$ and
 $\nu_q(d)=\nu_q(|b|_N)$. Since $|b|_M\in \mathcal{M}_b(M)$ then $|b|_N=q|b|_M\in
 \mathcal{M}_b(zN)$.
\end{enumerate}

\item Assume that $q=2$ and $b=2^{\gamma}-1$  for some
positive integer $\gamma\geq 2$.

We know, from Lemma \ref{motose}, that we can not find a prime $p$
such that $|b|_{p}=2$. So we follow a different procedure in this
case.

It is clear that $|b|_{q}=|b|_2=1$. Let $s=\nu_{2}(|b|_N)$ and
$c=\nu_2(N)$. Note that $c$ can not be strictly greater than $s$,
because $2$ divides $(b^k-1,N)$ and $\mathcal{M}_b(N)\neq
\varnothing$. We study the following cases:

\begin{enumerate}
\item $c=s$

By the assumption $c>0$. Suppose that there exists a $d\in
\mathcal{M}_b(N)$ such that $k$ is even. Thus $\nu_2(d)<s$. As $2$
divides \linebreak $(b^k-1,N)$ from Theorem \ref{ppl2} we have that
$c=\nu_2(N)\leq \nu_2(d)$ which is a contradiction. Therefore, it is
enough to take $z=1$.

\item $s>c$

In this case we take $z=2^{s-c}$. Since $|b|_{2^s}$ divides
$2^{s-1}$, then $|b|_{zN}=[|b|_{2^{s}},|b|_M]=|b|_M=|b|_N$. Hence,
$\mathcal{M}_b(zN)=\{d\in\mathcal{M}_b(N):|b|_N=kd \text{ and }
\nu_2(d)=\nu_2(|b|_N)\}.$

Indeed, from Theorem \ref{ppl2} we have that $d\in\mathcal{M}_b(N)$
is an element of $\mathcal{M}_b(zN)$ if and only if $s=\nu_2(zN)\leq
\nu_2(d)$ and this is equivalent to say that $\nu_2(d)=s$. Since
$|b|_N\in \mathcal{M}_b(N)$ and $s=\nu_2(|b|_N)$, we have that
$|b|_N\in \mathcal{M}_b(zN)$.
\end{enumerate}
\end{enumerate}
\end{proof}

\begin{theorem}
Let $N$ and $b$ be integers such that  $\left\vert
\mathcal{M}_{b}(N)\right\vert
>1$. Then, there
exists a positive integer $z$ such that $\mathcal{M}%
_{b}(zN) =\{|b|_N\}$.
\end{theorem}

\begin{proof} Let
$\left\vert b\right\vert _{N}=q_{1}^{t_{1}}\ldots q_{l}^{t_{l}}$ be
the prime factorization of $|b|_N$.

Applying Lemma \ref{conditions} to $q_1$ and $N$ we can find a
positive integer $z_1$ such that $|b_{z_1N}|=|b|_N$,
$\mathcal{M}_b(z_1N)\neq \varnothing$ and when $d\in
\mathcal{M}_b(z_1N)$, then $\nu_{q_1}(d)=\nu_{q_1}(|b|_N)$. Again
using Lemma \ref{conditions} with $q=q_2$ and $z_1N$, we get a
positive integer $z_2$ such that $|b|_{z_1z_2N}=|b|_N$,
$\mathcal{M}_b(z_1z_2N)\neq \varnothing$, and  $d\in
\mathcal{M}_b(z_1z_2N)$, implies that
$\nu_{q_2}(d)=\nu_{q_2}(|b|_N)$. From Theorem \ref{john} we know
that $\mathcal{M}_b(z_1z_2N)\subseteq \mathcal{M}_b(z_1N)$. In this
way for each $d\in \mathcal{M}_b(z_1z_2N)$ we also have that
$\nu_{q_1}(d)=\nu_{q_1}(|b|_N)$.

Repeating this process we get positive integers $z_1,\ldots,z_l$
such that if $z=\prod_{i=1}^l{z_i}$, the following properties hold
\begin{enumerate}
\item $|b|_{zN}=|b|_N$,
\item $\mathcal{M}_b(zN)\neq \varnothing$,
\item If $d\in \mathcal{M}_b(zN)$, then $\nu_{q_i}(d)=\nu_{q_i}(|b|_N)$
for all $i\in \{1,\ldots,l\}$.
\end{enumerate}

Since the $q_i$'s are the prime factors of  $|b|_N$, we conclude
that $d=|b|_N$ and therefore $\mathcal{M}_b(zN)=\{|b|_N\}$.
\end{proof}

\section*{Acknowledgements}
The authors are members of the research group: \'Algebra, Teor\'ia
de N\'umeros y Aplicaciones, ERM. J.H. Castillo was partially
supported by CAPES, CNPq from Brazil and Universidad de Nariño from
Colombia. J.M. Velásquez-Soto was partially supported by CONICET
from Argentina and Universidad del Valle from Colombia.

\nocite{Martin}

\bibliographystyle{amsalpha}
\bibliography{bibliografiaggp}

\end{document}